\documentclass[12pt,a4paper, reqno]{amsart}
\usepackage[utf8]{inputenc}
\usepackage[english]{babel}
\usepackage{amsmath,mathrsfs,amsfonts,amssymb,amsxtra,latexsym,amscd,amsthm,marvosym,multirow,mathtools, xcolor,yhmath,dsfont}
\usepackage[inner=2cm, outer=2cm, top=2cm, bottom=2.3cm]{geometry}
\usepackage{eucal}
\usepackage{enumerate}
\usepackage[shortlabels]{enumitem}
\usepackage{csquotes}
\usepackage{hyperref}
\usepackage{subcaption}

\usepackage{cleveref}

\newtheorem{theorem}{Theorem}[section]
\newtheorem{lemma}[theorem]{Lemma}

\newtheorem{proposition}[theorem]{Proposition}
\newtheorem{corollary}[theorem]{Corollary}

\theoremstyle{remark}

\renewcommand{\Re}{\operatorname{Re}\,}
\renewcommand{\Im}{\operatorname{Im}\,}

		\newcommand{\N}{\mathbb{N}}
		
		\newcommand{\R}{\mathbb{R}}
		\newcommand{\C}{\mathbb{C}}
		\newcommand{\A}{\textbf{A}}
		
		\newcommand{\dd}{\mathrm{d}}

		\newcommand{\bb}{\textbf{b}}

\newcommand{\curl}{\mathop{\mathrm{curl}}\nolimits}	
		
		\makeatletter
		
		\@addtoreset{equation}{section}  
		\makeatother
		
\begin{document}
			\title[]{Abrupt changes in the spectra of
			the Laplacian with constant complex magnetic field}
			%---------------David Krejcirik--------------------
			\author{David Krej\v{c}i\v{r}\'{i}k}
			\address[David Krej\v{c}i\v{r}\'{i}k]{Department of Mathematics, Faculty of Nuclear Sciences and Physical Engineering, Czech Technical University in Prague, ul. Trojanova 13/339, 12000 Prague, Czech Republic.}
			\email{david.krejcirik@fjfi.cvut.cz}
			%---------------Tho Nguyen Duc--------------------
			\author{Tho Nguyen Duc}
			\address[Tho Nguyen Duc]{Department of Mathematics, Faculty of Nuclear Sciences and Physical Engineering, Czech Technical University in Prague, ul. Trojanova 13/339, 12000 Prague, Czech Republic.}
			\email{thonguyen.khtn11@gmail.com}
			
			%---------------Nicolas Raymond--------------------
			\author{Nicolas Raymond}
			\address[Nicolas Raymond]{Univ Angers, CNRS, LAREMA, Institut Universitaire de France, SFR MATHSTIC, 49000 Angers, France}
			\email{nicolas.raymond@univ-angers.fr}
			%-----------------Abstract----------------------------------
			
			\date{23 December 2024}
			
			\begin{abstract}
We analyze the spectrum of the Laplace operator, subject to homogeneous complex magnetic fields in the plane. For real magnetic fields, it is well-known that the spectrum consists of isolated eigenvalues of infinite multiplicities (Landau levels). We demonstrate that when the magnetic field has a nonzero imaginary component, the spectrum expands to cover the entire complex plane. Additionally, we show that the Landau levels (appropriately rotated and now embedded in the complex plane) persists, unless the magnetic field is purely imaginary in which case they disappear and the spectrum becomes purely continuous. 
			\end{abstract}
			\maketitle
		%	\tableofcontents
		
%----------------------%		
\section{Introduction}
%----------------------%
%
\subsection{Motivation}
The spectrum of the Laplacian on the Euclidean plane is purely continuous. 
In 1930, Landau~\cite{Landau30} perceived the striking fact that
turning on the homogeneous magnetic field makes the spectrum pure point.
Now the mathematical description is through the magnetic Laplacian
\begin{equation}\label{operator}
  \mathscr{L}_\bb = 
  (-i\nabla -\A)^2
  \qquad \mbox{in} \qquad
  L^2(\R^2)
\end{equation}
with the vector potential  $\A(x)=\bb \, (0,x_{1})$
generating the constant magnetic field $\curl \A = \bb \in \R$.
Unless $\bb = 0$,
the spectrum of~$\mathscr{L}_\bb$ is composed of
isolated eigenvalues of infinite multiplicities
called \emph{Landau levels}
(see (A) versus (B) in Figure~\ref{Figure Spec}).
The quantization of spectra has been fundamental
in describing magnetic phenomena in physics,
in particular the quantum Hall effect.
We refer to 
\cite{Avron-Herbst-Simon_1978,CFKS-2nd,Mohamed-Raikov,
Erdos_2007,Fournais-Helffer_2009,Raymond-book}
for mathematically oriented surveys. 

Because of the quantum-mechanical purposes,
the almost century-long history has been restricted 
to real-valued magnetic fields.
Recent years, however, have brought new physical motivations 
for considering the Laplacian with complex magnetic fields.
Among these, there are superconductors, quantum statistical physics, 
stability of black holes in general relativity 
and the new concept of quasi-self-adjointness in quantum theory.
We refer to our preceding papers~\cite{K13,KNR24} and references therein.
Mathematically, the complexification involves 
the necessary passage to the unexplored realm 
of non-self-adjoint Schr\"odinger operators
with complex-valued vector potentials. 

The case of constant magnetic field was left untouched 
in our precedent paper~\cite{KNR24} for two reasons. 
First, the unboundedness of the vector potential~$\A$
leads to technical difficulties 
that we have been able to overcome only now. 
Second, and more importantly, the effect of complexifying 
the uniform magnetic field leads to spectrally striking phenomena 
that do not exist in the case of local fields considered in~\cite{KNR24}.
Indeed, the objective of the present paper is to show that 
the spectrum of~$\mathscr{L}_\bb$ with $\bb\in\C$
becomes the whole complex plane 
unless~$\bb$ is real.
What is more, the complexification 
has little effect on the Landau levels 
(they are just rotated in the complex plane),
unless~$\bb$ is purely imaginary  
in which case they disappear completely
 (see Figure~\ref{Figure Spec}).
In summary, the purely real or purely imaginary homogeneous magnetic fields 
are the only realizations for which the magnetic Laplacian 
possesses a purely continuous spectrum.  

\begin{figure}[ht!]
     \centering
     \begin{subfigure}[c]{0.19\textwidth}
         \centering
         \includegraphics[width=\textwidth]{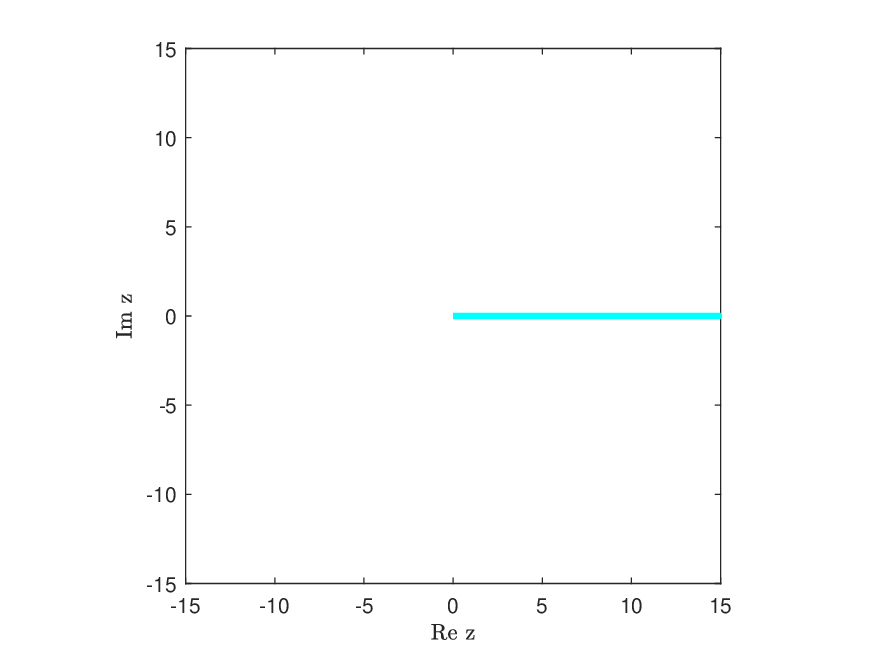}
         \caption{$\bb=0$.}
         \label{b=0}       
     \end{subfigure}
     \hfill
     \begin{subfigure}[c]{0.19\textwidth}
         \centering
         \includegraphics[width=\textwidth]{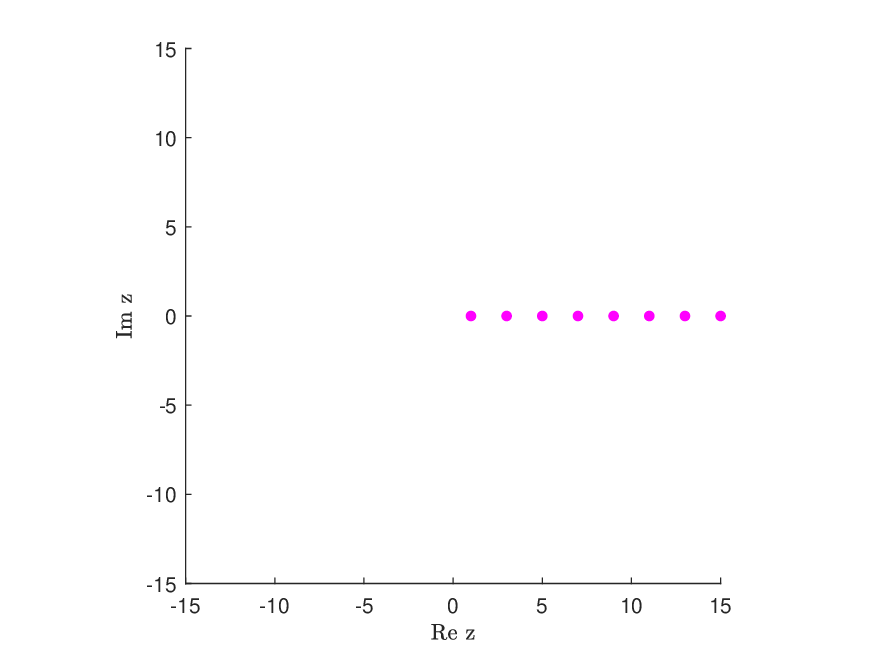}
         \caption{$\bb=\pm 1$.}\label{b=1}      
     \end{subfigure}
     \hfill
     \begin{subfigure}[c]{0.19\textwidth}
         \centering
         \includegraphics[width=\textwidth]{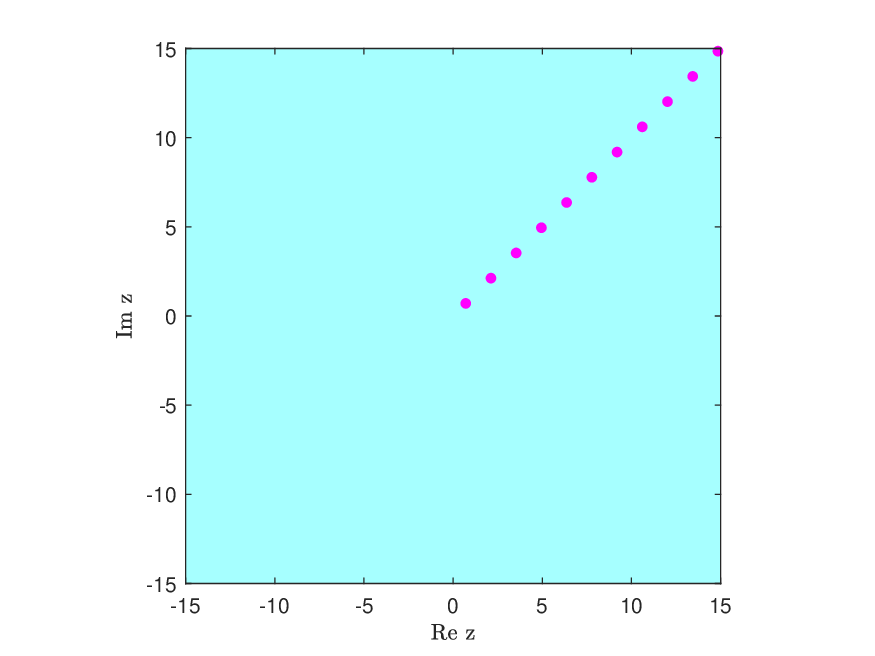}
         \caption{$\bb= \pm e^{i\pi/4}$.}     
     \end{subfigure}
     \hfill
     \begin{subfigure}[c]{0.19\textwidth}
         \centering
         \includegraphics[width=\textwidth]{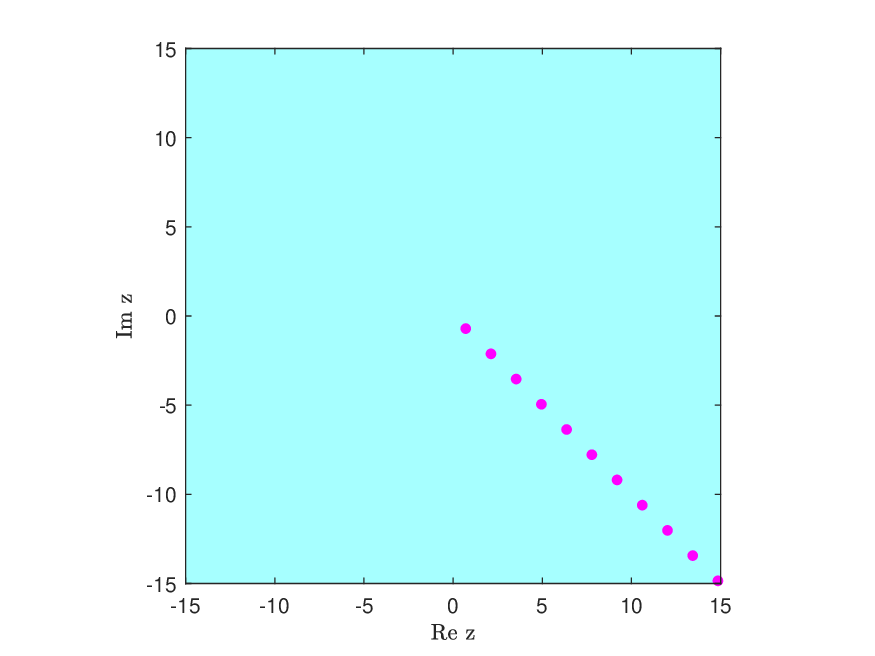}
         \caption{$\bb = \pm e^{-i\pi/4}$.}       
     \end{subfigure}     
     \hfill
     \begin{subfigure}[c]{0.19\textwidth}
         \centering
         \includegraphics[width=\textwidth]{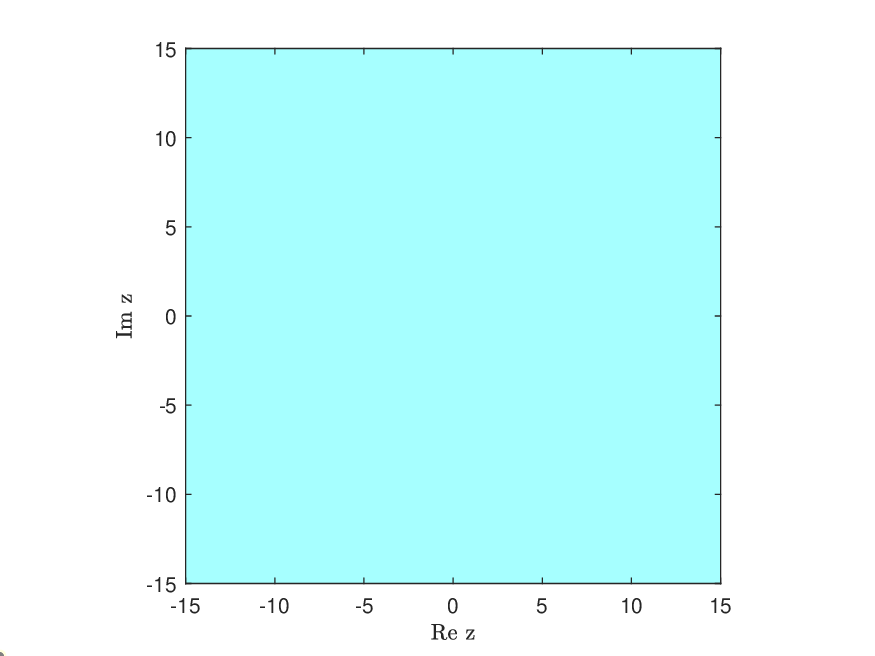}
         \caption{$\bb = \pm i$.}         
     \end{subfigure}
        \caption{The spectrum of the operator $\mathscr{L}_{\bb}$ for various values of $\bb$. The continuous spectrum is depicted in cyan, while the point spectrum is shown in magenta.}
        \label{Figure Spec}
     \end{figure}

\subsection{Main results}
Let us now describe the content of this paper in more detail.
We are interested in the maximal realization of~\eqref{operator},
where $\bb\in\C$.
More specifically, we consider
\begin{equation}\label{Operator}
	\begin{aligned}
		\mathscr{L}_\bb =& (-i\partial_{x_{1}})^2 + (-i\partial_{x_{2}}-\bb x_{1})^2\,,\\
		\operatorname{Dom}(\mathscr{L}_\bb)=&\left\{\psi\in L^2(\R^2) :\left[(-i\partial_{x_{1}})^2 + (-i\partial_{x_{2}}-\bb x_{1})^2\right]\psi\in L^2(\R^2)
		\right\}\,.
	\end{aligned}
\end{equation}

When $\bb$ is real, the operator $\mathscr{L}_{\bb}$ can be defined by using the Lax--Milgram theorem on an appropriate magnetic Sobolev space. However, when $\bb$ is not real, this approach is not applicable due to the lack of coerciveness of the corresponding quadratic form (see \cite[Ex.~1]{KNR24}). Consequently, there is no standard tool to compute the adjoint of $\mathscr{L}_{\bb}$ in this case.  Moreover, in (magnetic) Sobolev spaces, the density of $C_{c}^{\infty}(\R^2)$ can be established by employing mollifiers together with an expanding cutoff sequence. This relies on the assumption that first-order (covariant) derivatives belong to $L^2(\R^2)$. However, when $\bb$ is complex, this property no longer holds in $\operatorname{Dom}(\mathscr{L}_{\bb})$, as the operator is not essentially self-adjoint. Consequently, proving the density of $C_{c}^{\infty}(\R^2)$ becomes unattainable in the usual sense. To overcome this difficulty, in Section~\ref{Subsec weak core}, we introduce a novel approach, which we call the \emph{weak core method}, to establish that $C_{c}^{\infty}(\R^2)$ forms a weak core of the domain. This result enables us to compute the adjoint of the maximal operator $\mathscr{L}_{\bb}$ directly from its definition. In this way, we manage to show that $\mathscr{L}_\bb$ is well defined
for any $\bb \in \C$.
\begin{proposition}\label{Theo Adjoint}
$\mathscr{L}_{\bb}$ is a closed, densely defined operator and its adjoint is given by 
\begin{equation}\label{Csa}
  \mathscr{L}_{\bb}^{*}
  = \mathscr{L}_{\overline{\bb}}
  = C\mathscr{L}_{\bb} C^{-1}
  \,,
\end{equation}
where
$$
C: L^2(\R^2) \to L^2(\R^2),\qquad (C\psi)(x_{1},x_{2})=\overline{\psi(-x_{1},x_{2})}\,.$$
Consequently, $\mathscr{L}_{\bb}$ is self-adjoint if and only if $\bb$ is real.  
\end{proposition}

Since~$\mathscr{L}_{\bb}$ is closed, 
the spectrum can be decomposed as
 \[ 
 \operatorname{Spec}(\mathscr{L}_{\bb})
 = \operatorname{Spec}_{\textup{p}}(\mathscr{L}_{\bb}) 
 \cup \operatorname{Spec}_{\textup{c}}(\mathscr{L}_{\bb})
 \cup \operatorname{Spec}_{\textup{r}}(\mathscr{L}_{\bb})
 \,,
 \]
where the disjoint sets on the right-hand side denote 
the point, continuous and residual spectra, respectively 
(to recall the standard definitions, see Section~\ref{Sec.notation}).
The relationship~\eqref{Csa} reveals that~$\mathscr{L}_{\bb}$ 
is complex-self-adjoint with respect to the conjugation~$C$.
Consequently (see \cite[Prop.~1]{Camara-Krejcirik23}),
its residual spectrum is always empty:
$\operatorname{Spec}_{\textup{r}}(\mathscr{L}_{\bb})=\emptyset$.  
 
By taking the partial Fourier transform~$\mathcal{F}$ in the $x_2$-variable
and the complex change of variable $y_1 := x_1 - \bb^{-1} \xi_2$ 
where $\xi_2 = \mathcal{F} (-i\partial_{x_2}) \mathcal{F}^{-1}$,
the operator~$\mathscr{L}_{\bb}$ is formally similar to 
the complex-rotated harmonic oscillator 
(originally due to~\cite{Davies_1999a}
and recently revised in~\cite{Arnal-Siegl_2023})
\begin{equation}\label{Davies}
  -\partial_{y_{1}}^2 + \bb^2 \, y_1^2
  \qquad \mbox{in} \qquad 
  L^2(\R^2)
  \,.
\end{equation}
Of course, this procedure is purely formal (unless~$\bb$ is real), 
but it naturally leads to the definition of \emph{complex Landau levels}
\begin{equation}\label{Landau levels}
\Lambda_{\bb}= \{\pm (2k+1)\bb: k\in \N_{0}\}\, \qquad \text{if }\pm\Re \bb> 0
\,, 
\end{equation}
being the eigenvalues of~\eqref{Davies}.
Also, it is clear that the eigenvalues are infinitely degenerate
(for the variable~$y_2$ is missing in the action of~\eqref{Davies}).
 
It is well known that 
$
  \operatorname{Spec}(\mathscr{L}_{\bb}) 
  = \operatorname{Spec}_{\textup{c}}(\mathscr{L}_{\bb})
  = [0,+\infty)
$
if $\bb = 0$. 
At the same time, 
$
  \operatorname{Spec}(\mathscr{L}_{\bb}) 
  = \operatorname{Spec}_{\textup{p}}(\mathscr{L}_{\bb})
  = \Lambda_{\bb}
$
if $\bb \in \R \setminus \{0\}$. 
In the following theorem, 
we characterize the spectrum in the unexplored situations.

				\begin{theorem}\label{Theo Spectrum}
				\begin{enumerate}[label=\textup{(\roman*)}]
					\item \label{b complex} When $\bb \in \C\setminus \left(\R \cup i\R\right)$, the spectrum is the whole complex plane, including the complex Landau levels as the only eigenvalues:
					\[ \operatorname{Spec}(\mathscr{L}_\bb) =\C, \qquad\operatorname{Spec}_{\textup{c}}(\mathscr{L}_\bb)=\C\setminus \Lambda_{\bb},\qquad \operatorname{Spec}_{\textup{p}}(\mathscr{L}_\bb)= \Lambda_{\bb}\,.\]
					\item \label{b imaginary} When $\bb \in i\R \setminus \{0\} $, 
					the spectrum is the whole complex plane and contains no eigenvalues:
					\[
						\operatorname{Spec}(\mathscr{L}_\bb)=\C,\qquad \operatorname{Spec}_{\textup{c}}(\mathscr{L}_\bb)= \C, \qquad \operatorname{Spec}_{\textup{p}}(\mathscr{L}_\bb)=\emptyset\,.
					\]
					
				\end{enumerate}
				In both cases, all types of essential spectra are identical to the spectrum
				\[ \operatorname{Spec}_{\textup{ess,k}}(\mathscr{L}_\bb)=\operatorname{Spec}(\mathscr{L}_\bb),\qquad \forall \textup{k} \in \{1,\dots,5\} \,.\]
			\end{theorem}
			
Since $\mathscr{L}_{\bb}$ is complex-self-adjoint, it follows that all the sets $\operatorname{Spec}_{\textup{ess,k}}(\mathscr{L}_{\bb})$ are identical for 
$k\in \{1,\dots,4\}$ 
(cf.~\cite[Thm.~9.1.6(ii)]{Edmunds-Evans18}). In the proof of Theorem~\ref{Theo Spectrum}, 
we demonstrate that $\operatorname{Spec}_{\textup{ess,2}}(\mathscr{L}_{\bb})=\C$, which implies the equivalence of the remaining essential spectra, including the fifth one. Illustrations of the spectra for selected values of $\bb$ are presented in Figure \ref{Figure Spec}. 

From Theorem~\ref{Theo Spectrum}, it is evident that the numerical range of $\mathscr{L}_{\bb}$ encompasses the entire complex plane when $\bb$ is not real. This explains why~$\mathscr{L}_{\bb}$ can not be realized as an m-sectorial operator, as discussed in \cite[Ex.~1]{KNR24}. Nevertheless, an intriguing observation is that the point spectrum always resides within the right-hand half-plane
(unless~$\bb$ is purely imaginary).

We emphasize that the established results are obtained 
for a special choice of the vector potential $\A(x)=\bb \,  (0,x_1)$
in~\eqref{operator}
generating the constant magnetic field $\curl \A = \bb$. 
Unlike the self-adjoint case (when $\bb$ is real), 
the results do not automatically extend to other choices of~$\A$ 
satisfying $\curl \A = \bb$.
This is due to the lack of gauge invariance in the non-self-adjoint setting.
It is interesting to explore how the spectrum appears 
for other choices of the magnetic potential~$\A$ satisfying $\curl\A = \bb$,
in particular for  the transverse gauge 
$\A(x)= \frac{1}{2} \bb \, (-x_{2},x_{1})$.

\subsection{General notations}\label{Sec.notation}
Let us fix some notations employed throughout the paper.
\begin{enumerate}[label=\textup{(\arabic*)}]
\item The inner product on $L^2(\R^2)$ is denoted by $\langle \cdot, \cdot \rangle$ and we use $\Vert \cdot \Vert$ for $L^2$-norm.
\item The characteristic function of any subset $E$ of $\R$ is denoted by $\mathds{1}_{E}$.
\item \label{Not Spec} For a linear operator $\mathscr{L}$, we denote its kernel, range and spectrum, respectively, by $\operatorname{Ker}(\mathscr{L})$, $\operatorname{Ran}(\mathscr{L})$ and $\operatorname{Spec}(\mathscr{L})$. 
The point, continuous and residual spectra are, respectively, defined by
\begin{align*}
&\operatorname{Spec}_{\textup{p}}(\mathscr{L})=  \{z\in \C: \mathscr{L}-z \text{ is not injective} \}\,,\\
&\operatorname{Spec}_{\textup{c}}(\mathscr{L})=  \{z\in \C: \mathscr{L}-z \text{ is injective and } \overline{\textup{Ran}(\mathscr{L}-z)}=\mathcal{H} \text{ and } \textup{Ran}(\mathscr{L}-z)\subsetneq \mathcal{H} \} \,, \\
&\operatorname{Spec}_{\textup{r}}(\mathscr{L})=  \{z\in \C: \mathscr{L}-z \text{ is injective and } \overline{\textup{Ran}(\mathscr{L}-z)} \subsetneq \mathcal{H} \}
\,.
\end{align*}
We also recall here five types of essential spectra of a closed operator $\mathscr{L}$ spectra 
as defined in \cite[Sec.~5.4.2]{Krejcirik-Siegl15}:
\begin{align*}
&\operatorname{Spec}_{\textup{ess,1}}(\mathscr{L})= \C\setminus \{z\in \C: \mathscr{L}-z \text{ is semi-Fredholm} \}\,,\\
&\operatorname{Spec}_{\textup{ess,2}}(\mathscr{L})=  \C \setminus \{z\in \C: \operatorname{Ran}(\mathscr{L}-z) \text{ is closed and} \operatorname{dim}\operatorname{Ker}\left(\mathscr{L}-z\right) \text{ is finite} \}\,,\\
&\operatorname{Spec}_{\textup{ess,3}}(\mathscr{L})= \C\setminus \{z\in \C: \mathscr{L}-z \text{ is Fredholm}\}\,,\\
&\operatorname{Spec}_{\textup{ess,4}}(\mathscr{L})= \C\setminus \{z\in \C: \mathscr{L}-z \text{ is Fredholm with index 0}\}\,,\\
&\operatorname{Spec}_{\textup{ess,5}}(\mathscr{L})= \C\setminus \left\{z\in \C: z \text{ is an isolated eigenvalue with finite algebraic multiplicity} \right.\\
&\hspace{5.2cm}\left. \text{and }\operatorname{Ran}(\mathscr{L}-z) \text{ is closed}  \right\}\,.
\end{align*}
\item For the multi-valued exponential function $z^c$ where $z,c\in \C$, we choose its principal value and still denote it as $z^c$, \emph{i.e.},
				$ z^{c} = e^{c \textup{ Log }z},$
where $\textup{Log }z = \log |z| + i \textup{ Arg }z$. Then, we have
				\begin{equation}\label{Power Modul}
					|z^c| = e^{-(\operatorname{Arg} z)(\Im c )} |z|^{\Re c}\,.
				\end{equation}
\item For two real-valued functions $a$ and $b$, we write $a\lesssim b$ (respectively, $a\gtrsim b$) instead of $a\leq C b$ (respectively, $a \geq C b$) for an insignificant constant $C>0$. We write $a\approx b$ when $a\lesssim b$ and $a\gtrsim b$.
\item The Fourier transform is given by
\begin{equation}\label{Fourier}
\mathcal{F}_{x\mapsto \xi}(\psi)(\xi) = \frac{1}{2\pi} \int_{\R^2} e^{-i x\cdot \xi} \psi(x) \dd x\,.
\end{equation}
The partial Fourier transform in the second variable is given by
\begin{equation}
\mathcal{F}_{x_{2}\mapsto \xi_{2}}(\psi)(\xi) = \frac{1}{\sqrt{2\pi}} \int_{\R} e^{-i x_{2} \xi_2} \psi(x)\, \dd x_{2}\,.\label{Partial Fourier 2}
\end{equation}
They are unitary on $L^2(\R^2)$.
\end{enumerate}
\subsection{Structure of the paper}	
The paper is organized as follows: In Section~\ref{Sec CSA}, we show that the operator $\mathscr{L}_{\bb}$ is complex-self-adjoint and present several equivalent spectra for operators corresponding to different values of $\bb$. The spectral analysis of the operator is investigated in 
Sections~\ref{Sec Complex b} and~\ref{Sec Imag b} 
as regards the cases~(i) and~(ii) of Theorem~\ref{Theo Spectrum}, respectively.

%-----------------------------------------------%
\section{Definition of the magnetic Laplacian}\label{Sec CSA}
%-----------------------------------------------%
%
The main objectives of this section are to prove Propositions \ref{Prop Adjoint} and to demonstrate that the analysis for all $\bb$ can be reduced to $\bb$ lying on the first-quadrant arc of the unit circle. 

As usual, we understand the action of $\mathscr{L}_\bb$
in~\eqref{Operator} in the sense of distribution. 
It means that $\psi\in \operatorname{Dom}(\mathscr{L}_\bb)$ if and only if $\psi\in L^2(\R^2)$ and there exists $f\in L^2(\R^2)$ such that
\begin{equation}\label{Distribution}
\left\langle \psi, \left[(-i\partial_{x_{1}})^2 + (-i\partial_{x_{2}}-\overline{\bb} x_{1})^2\right]\varphi \right\rangle = \left\langle f, \varphi \right\rangle\qquad \forall \varphi \in C_{c}^{\infty}(\R^2)\,,
\end{equation}
and we denote $\mathscr{L}_{\bb}\psi =f$.

\subsection{A weak core result}\label{Subsec weak core}
In order to find the adjoint of the maximal operator $\mathscr{L}_{\bb}$, we need the following \enquote{weak core} result.
			\begin{proposition}\label{Prop Adjoint}
			Let $\bb\in \C$. For any $\psi\in \operatorname{Dom}(\mathscr{L}_{\bb})$, 
			there exists a sequence $(\psi_{n}) \subset C_{c}^{\infty}(\R^2)$ such that $\psi_{n}\xrightarrow[n\to +\infty]{} \psi $ in $L^2(\R)$ and  $\mathscr{L}_{\bb}\psi_{n} \xrightarrow[n\to +\infty]{} \mathscr{L}_{\bb} \psi$ weakly in $L^2(\R^2)$.
			
%			\begin{align*}
%				\mathscr{L}_\bb^{*} &= (-i\partial_{x_{1}})^2 + (-i\partial_{x_{2}}-\overline{\bb} x_{1})^2\\
%				\operatorname{Dom}(\mathscr{L}_\bb^{*}) &=\{\psi\in L^2(\R^2) : \left[(-i\partial_{x_{1}})^2 + (-i\partial_{x_{2}}-\overline{\bb} x_{1})^2\right]\psi\in L^2(\R^2)\}.
%			\end{align*}
%			Therefore, the operator $\mathscr{L}_{\bb}$ is self-adjoint if and only if $\bb\in \R$.
			
			\end{proposition}
			\begin{proof}				
We consider two cut-off functions $\chi$ and $\rho$ in $\C_{c}^{\infty}(\R^2)$ such that 
\begin{itemize}
\item[$\bullet$] $0\leq \chi \leq 1$, $  \operatorname{Supp}\chi \subset B(0,1)$ and $\chi=1$ on $ B(0,\frac{1}{2})$,
\item[$\bullet$] $\rho\geq 0$, $  \operatorname{Supp}\rho \subset B(0,1)$ and $\int \rho(x)\, \dd x=1$.
\end{itemize}					 
We then define the usual \emph{expanding cut-offs} and \emph{shrinking mollifiers} as follows:
					\[\chi_{n}(\cdot) = \chi\left(\frac{\cdot}{n}\right)\,,\qquad \rho_{n}(\cdot) =  n^2\rho(n \cdot)\,.\]
For $\psi \in \operatorname{Dom}(\mathscr{L}_\bb)$, we will show that $\psi_{n}= \chi_{n}(\rho_{n} \ast \psi)$ satisfies our requirements.  It is known that $\psi_{n}\in C_{c}^{\infty}(\R^2)$ 
(cf.~\cite[Prop.~4.20]{Brezis11}) and $\psi_{n} \xrightarrow[n\to+\infty]{} \psi$ in $L^2(\R^2)$ (cf. \cite[Lem.~1.9]{CR21}). To prove that $ \mathscr{L}_{\bb} \psi_{n}$ converges weakly to $\mathscr{L}_{\bb} \psi$ in $L^2(\R^2)$, we write
					\begin{align*}
						\mathscr{L}_{\bb} \psi_{n}-\mathscr{L}_{\bb} \psi= I_{n}+J_{n} +K_{n}\,,
					\end{align*}
					where
					\begin{align*}
						I_{n}= &\mathscr{L}_{\bb} (\chi_{n} \rho_{n} \ast \psi)- \chi_{n} \mathscr{L}_{\bb}(\rho_{n} \ast \psi)\, ,\\
						J_{n}= & \chi_{n} \left[\mathscr{L}_{\bb} (\rho_{n} \ast \psi)- \rho_{n} \ast (\mathscr{L}_{\bb} \psi) \right]\,,\\
						K_{n} = & \chi_{n}(\rho_{n} \ast \mathscr{L}_{\bb} \psi)-\mathscr{L}_{\bb} \psi\,.
					\end{align*}
					As $\mathscr{L}_{\bb} \psi\in L^2(\R^2)$, 
					$K_n \xrightarrow[n\to+\infty]{}0$ in $L^2(\R^2)$. 
					Below, we evaluate $\mathscr{L}_{b}$ only for smooth functions, 
					using the expansion
					\begin{equation}\label{Expression Lb}
					\mathscr{L}_{\bb}= -\Delta +2i\bb\, x_{1} \partial_{x_{2}}+\bb^2 x_{1}^2\,.
					\end{equation}
					For $I_{n}$, using the chain rule, we have
					\[ I_{n}=\frac{1}{n^2} \Delta \chi \left(\frac{\cdot}{n}\right)(\rho_{n} \ast \psi)+2\frac{1}{n}\nabla \chi\left(\frac{x}{n}\right)\cdot \nabla (\rho_{n} \ast \psi)\,.\]
					Since $\operatorname{Supp} (\rho_{n} \ast \psi) \subset \operatorname{Supp} \psi + \overline{B(0,1)} $ is fixed and the derivatives of 
					$x \mapsto \chi \left( \frac{x}{n} \right)$ vanish for $|x|\leq \frac{n}{2}$, it follows that $I_n=0$ for sufficiently large $n$.
					
					Now, we will show that $J_{n}$ converges weakly to zero. 
					Given $\eta\in C_{c}^{\infty}(\R^2)$, thanks to integration by parts and \cite[Prop.~4.16]{Brezis11}, we have
					\begin{equation}\label{Eq 1}
						\langle \mathscr{L}_{\bb} (\rho_{n} \ast \psi), \eta \rangle =  \left\langle \rho_{n} \ast \psi, \mathscr{L}_{\overline{\bb}} \eta\right\rangle=\left\langle  \psi, \check{\rho}_{n} \ast \mathscr{L}_{\overline{\bb}} \eta\right\rangle\,,
					\end{equation}
					where $\check{\rho}_n(x)=\rho_{n}(-x)$.\\
					By the linearity of convolution, we get
					\begin{align*}
						\check{\rho}_{n} \ast \mathscr{L}_{\overline{\bb}} \eta = \check{\rho}_{n} \ast (-\Delta \eta)  +2i \overline{\bb}\, \check{\rho}_{n} \ast (x_1 \partial_{x_{2}} \eta) +\overline{\bb}^2 \, \check{\rho}_{n} \ast (x_{1}^2 \eta)\,.
					\end{align*}
					From \cite[Prop.~4.20]{Brezis11}, which establishes the commutative property of mollifiers with the derivatives, and the definition of the convolution, we see that
					\begin{align*}
						\check{\rho}_{n} \ast (-\Delta \eta) = & -\Delta \left(\check{\rho}_{n} \ast \eta\right)\,,\\
						\check{\rho}_{n} \ast (x_1 \partial_{x_{2}}\eta) =& x_{1} \partial_{x_{2}} \left(\check{\rho}_{n} \ast \eta\right) - \left((x_{1} \check{\rho}_{n}) \ast \partial_{2}\eta\right)\,,\\
						\check{\rho}_{n} \ast (x_{1}^2 \eta)= & x_{1}^2  \left(\check{\rho}_{n} \ast\eta\right) - 2 x_{1}\left((x_{1} \check{\rho}_{n}) \ast \eta\right)+ \left((x_{1}^2 \check{\rho}_{n}) \ast \eta\right)\,.
					\end{align*}
					Combining these identities, we obtain
					\begin{equation}\label{Eq 2}
						\check{\rho}_{n} \ast \mathscr{L}_{\overline{\bb}}\eta = \mathscr{L}_{\overline{\bb}} \left(\check{\rho}_{n} \ast \eta\right) - R_{\eta}\,,
					\end{equation}
					where
					\[R_{\eta}= \overline{\bb}^2 \left((x_{1}^2 \check{\rho}_{n}) \ast \eta\right)- 2 \overline{\bb}^2  x_{1}\left((x_{1} \check{\rho}_{n}) \ast \eta\right) -2i\overline{\bb} \,\left((x_{1} \check{\rho}_{n}) \ast \partial_{2}\eta\right)\,. \]
					From \eqref{Eq 1} and \eqref{Eq 2}, we deduce that
					\begin{align*}
					\langle \mathscr{L}_{\bb} (\rho_{n} \ast \psi), \eta \rangle= &\langle \psi, \mathscr{L}_{\overline{\bb}} \left(\check{\rho}_{n} \ast \eta\right)\rangle -\langle \psi, R_{\eta} \rangle\\
					= &\langle \mathscr{L}_{\bb}\psi,  \check{\rho}_{n} \ast \eta \rangle -\langle \psi, R_{\eta} \rangle\\
					= &\langle \rho_{n}\ast \mathscr{L}_{\bb}\psi,  \eta \rangle -\langle \psi, R_{\eta} \rangle\,.
					\end{align*}
			Here, in the second equality, noting that $\check{\rho}_{n} \ast \eta \in C_{c}^{\infty}(\R^2)$, we employed \eqref{Distribution} and in the third equality, as $\mathscr{L}_{\bb}\psi \in L^2(\R^2)$, we applied \cite[Prop.~4.16]{Brezis11} again. Therefore, thanks to Cauchy-Schwarz inequality, we have, for every $\eta\in C_{c}^{\infty}(\R^2)$,
					\begin{equation*}
						\left\vert \langle \mathscr{L}_{\bb} (\rho_{n} \ast \psi)- \rho_{n} \ast \mathscr{L}_{\bb} \psi, \eta \rangle\right\vert\leq \Vert \psi \Vert \Vert R_{\eta} \Vert\,.
					\end{equation*}	
					By choosing $\eta=\chi_{n}\varphi$ for a fixed but arbitrary $\varphi\in \C_{c}^{\infty}(\R^2)$, we have
					\begin{equation}\label{Est Jn}
						\left\vert \langle J_{n}, \varphi \rangle\right\vert\leq \Vert \psi \Vert \Vert R_{\chi_{n}\varphi} \Vert\,.
					\end{equation}	
					Let us estimate each term in $R_{\chi_{n}\varphi}$. 
					By using Young's inequality, we have
					\begin{align*}
					&\Vert (x_{1}^2 \check{\rho}_{n}) \ast (\chi_{n}\varphi) \Vert\leq \Vert x_{1}^2 \check{\rho}_{n} \Vert_{L^1} \Vert \chi_{n}\varphi \Vert\leq \frac{1}{n^2} \Vert x_{1}^2 \rho \Vert_{L^1} \Vert \varphi \Vert\,,\\
					&\Vert x_{1}\left[(x_{1} \check{\rho}_{n}) \ast (\chi_{n}\varphi)\right] \Vert\leq C \Vert (x_{1} \check{\rho}_{n}) \ast (\chi_{n}\varphi)\Vert\leq C \Vert x_{1} \check{\rho}_{n} \Vert_{L^1} \Vert \varphi \Vert \leq \frac{C}{n} \Vert x_{1} \rho \Vert_{L^1} \Vert \varphi \Vert\,,\\
					&\Vert (x_{1} \check{\rho}_{n}) \ast \partial_{2}(\chi_{n}\varphi) \Vert \leq \Vert x_{1}  \check{\rho}_{n} \Vert_{L^1} \Vert \partial_{2}(\chi_{n}\varphi) \Vert\leq \frac{1}{n} \Vert x_{1}  \rho \Vert_{L^1} \Vert \partial_{2}(\chi_{n}\varphi) \Vert\leq \frac{1}{n} \Vert x_{1}  \rho \Vert_{L^1} \Vert \partial_{2}\varphi \Vert\,.
					\end{align*}
						As $\operatorname{Supp} \left[(x_{1} \check{\rho}_{n}) \ast (\chi_{n}\varphi)\right]\subset \overline{B(0,1)}+\operatorname{Supp}\varphi$ for $n$ large enough, the constant $C$ depends only on the support of $\varphi$. From these bounds and from \eqref{Est Jn}, we deduce that $J_{n}$ converges weakly to zero in $L^2(\R^2)$ and thus, $ \mathscr{L}_{\bb} \psi_{n}$ converges weakly to $\mathscr{L}_{\bb} \psi$ in $L^2(\R^2)$.
\end{proof}
Now, we turn to the main task of this section.
\begin{proof}[Proof of Proposition \ref{Prop Adjoint}]
Clearly, $C_{c}^{\infty}(\R^2) \subset \operatorname{Dom}(\mathscr{L}_\bb)$, and since  $C_{c}^{\infty}(\R^2)$ is dense in $L^2(\R^2)$, it follows that $\mathscr{L}_{\bb}$ is densely defined.

 To prove that $\mathscr{L}_{\bb}$ is closed, let $\psi_{n}\in \operatorname{Dom}(\mathscr{L}_\bb)$ be a sequence such that $\psi_{n}\xrightarrow[n\to+\infty]{}\psi$ and $\mathscr{L}_{\bb} \psi_{n}\xrightarrow[n\to+\infty]{} f$ in $L^2(\R^2)$. We need to show that $\psi\in  \operatorname{Dom}(\mathscr{L}_\bb)$ and $  \mathscr{L}_\bb \psi=f$. By \eqref{Distribution}, for all $\varphi\in C_{c}^{\infty}(\R^2)$, we have,
\[ \langle \psi, \mathscr{L}_{\overline{\bb}} \varphi \rangle = \lim_{n\to+\infty} \langle \psi_{n}, \mathscr{L}_{\overline{\bb}} \varphi \rangle =\lim_{n\to+\infty} \langle  \mathscr{L}_{\bb}\psi_{n}, \varphi \rangle = \langle f, \varphi \rangle\,. \]					
This implies that $\psi\in \operatorname{Dom}(\mathscr{L}_\bb)$ and $f=\mathscr{L}_{\bb} \psi$, as desired. Thus, $\mathscr{L}_\bb$ is closed.

To determine the adjoint of $\mathscr{L}_{\bb}$, let $g\in \operatorname{Dom}(\mathscr{L}_{\bb}^{*})$. By definition, $g\in L^2(\R^2)$ and there exists $g^{*}\in L^2(\R^2)$ such that
					\[ \langle \mathscr{L}_{\bb} \varphi, g \rangle = \langle \varphi, g^{*}\rangle,\qquad \forall \varphi\in \operatorname{Dom}(\mathscr{L}_{\bb})\,.\]	
					By Proposition \ref{Prop Adjoint}, this is equivalent to
						\[ \langle \mathscr{L}_{\bb} \varphi, g \rangle = \langle \varphi, g^{*}\rangle,\qquad \forall \varphi\in C_{c}^{\infty}(\R^2)\,,\]
						which implies that
						$$g\in \{\psi\in L^2(\R^2) : \left[(-i\partial_{x_{1}})^2 + (-i\partial_{x_{2}}-\overline{\bb} x_{1})^2\right]\psi\in L^2(\R^2)\}\,,$$
						and $\mathscr{L}_{\bb}^{*}g=g^{*}=\mathscr{L}_{\overline{\bb}}g$. Thus, the adjoint of $\mathscr{L}_{\bb}$ is $\mathscr{L}_{\bb}^{*}=\mathscr{L}_{\overline{\bb}}$.
						
						Now, suppose $\bb$ is real. Then $\overline{\bb}=\bb$, so $\mathscr{L}_{\bb}=\mathscr{L}_{\bb}^{*}$, showing that $\mathscr{L}_{\bb}$ is self-adjoint. Conversely, assume $\mathscr{L}_{\bb}=\mathscr{L}_{\bb}^{*}$. Take any real-valued function $\psi \in C_{c}^{\infty}(\R^2)$. Using the explicit form \eqref{Expression Lb} of $\mathscr{L}_{\bb}$, we have
						\[ \mathscr{L}_{\bb} \psi - \mathscr{L}_{\overline{\bb}}\psi = -4\bb_{2} x_{1}\left( \partial_{x_{2}}\psi-i\bb_{1} x_{1} \psi\right)\,,\]
						where $\bb=\bb_{1}+i\bb_{2}$ with $\bb_{1},\bb_{2}\in \R$. Since $x_{1}\left( \partial_{x_{2}}\psi-i\bb_{1} x_{1} \psi\right)$ is a nonzero function, this implies that $\bb_{2}=0$, so $\bb$ must be real.
						
						Finally, we show that $\mathscr{L}_{\bb}$ is $C$-self-adjoint. Let $\varphi \in C_{c}^{\infty}(\R^2)$. A straightforward calculation verifies that $\left(\mathscr{L}_\bb C\right) \varphi  = \left(C \mathscr{L}_{\overline{\bb}} \right)  \varphi$. For any $\psi\in \operatorname{Dom}(\mathscr{L}_{\bb}^{*})$, we have
						\[ \langle C\psi, \mathscr{L}_{\overline{\bb}} \varphi \rangle=\langle C\psi, C\mathscr{L}_{\bb} C\varphi \rangle=\langle \mathscr{L}_{\bb} C \varphi, \psi \rangle= \langle C\varphi, \mathscr{L}_{\bb}^{*} \psi \rangle= \langle C \mathscr{L}_{\bb}^{*} \psi, \varphi \rangle\,.\]
						Here, the second and fourth equalities follows from the properties $\langle Cu, v\rangle=\langle Cv,u\rangle$ for any $u,v\in L^2(\R^2)$ and $C^2=I$. Therefore,
						\[ \langle C\psi, \mathscr{L}_{\overline{\bb}} \varphi \rangle =\langle C \mathscr{L}_{\bb}^{*} \psi, \varphi \rangle\,,\qquad \forall \varphi \in C_{c}^{\infty}(\R^2)\,. \]
						This implies that $C\psi \in  \operatorname{Dom}(\mathscr{L}_{\bb})$ and $\mathscr{L}_{\bb} C \psi= C \mathscr{L}_{\bb}^{*} \psi$. Hence, $\mathscr{L}_{\bb}^{*} =C \mathscr{L}_{\bb} C$, proving that $\mathscr{L}_{\bb}$ is complex-self-adjoint.
\end{proof}

\subsection{Spectral reducibility}\label{Subsec Spectra Reduc}
Below, we outline the equivalence of the spectra of the operator $\mathscr{L}_{\bb}$ for different values of the parameter $\bb$. These equivalences highlight the symmetry properties and scaling behavior of $\mathscr{L}_{\bb}$ in relation to its defining parameters, providing insights into the spectral characteristics of the system. 
\begin{proposition}\label{Prop Spectra}
				Given $\bb\in \C$, we write $\bb= |\bb| e^{i \theta}$ where $\theta= \operatorname{Arg}\bb$. Then,
				\begin{enumerate}[label=\textup{(\roman*)}]
					\item \label{Scalling} $ \operatorname{Spec}_{\tau}(\mathscr{L}_{\bb})=|\bb|\operatorname{Spec}_{\tau}(\mathscr{L}_{e^{i\theta}})$\,,
					\item \label{Sym 1} $ \operatorname{Spec}_{\tau}(\mathscr{L}_{e^{i\theta}})=\operatorname{Spec}_{\tau}(\mathscr{L}_{-e^{i\theta}})$\,,
					\item \label{Sym 2}$ \operatorname{Spec}_{\tau}(\mathscr{L}_{e^{i\theta}})=\overline{\operatorname{Spec}_{\tau}(\mathscr{L}_{e^{-i\theta}})}$\,,
				\end{enumerate}
				where $\tau\in \left\{\,; \textup{r}; \textup{p}; \textup{c}; \textup{ess,2}; \textup{ess,5}\right\}$.
			\end{proposition}
			\begin{proof}
			The proof is a direct application of \cite[Prop.~5.5.1]{Krejcirik-Siegl15}, utilizing suitable (anti-)unitary operators. In particular, \ref{Scalling} and \ref{Sym 1} are deduced, respectively, from 
			$$V_{\bb}^{-1}\mathscr{L}_\bb V_\bb=|\bb|\mathscr{L}_{|\bb|^{-1}\bb},\qquad S^{-1}\mathscr{L}_\bb S=\mathscr{L}_{-\bb}\,, $$
			where $V_{\bb}$ and $S$ are unitary operators defined by
			\begin{align*}
			V_{\bb} &: L^2(\R^2) \to L^2(\R^2),\qquad V_\bb\psi(x)=|\bb|^{\frac12}\psi(|\bb|^{\frac12}x)\,,\\
			S &: L^2(\R^2) \to L^2(\R^2),\qquad  (S\psi)(x_{1},x_{2})=\psi(-x_{1},x_{2})\,.
			\end{align*}
				While \ref{Sym 2} is ensured from the complex-self-adjointness of $\mathscr{L}_{\bb}$, $\mathscr{L}_{\overline{\bb}}=C \mathscr{L}_{\bb} C^{-1}$.
			\end{proof}
			
As a consequence of	Proposition~\ref{Prop Spectra}, 
for $\bb\in \C \setminus \{0\}$,  the spectral analysis of $\mathscr{L}_{\bb}$ can be systematically reduced as follows:
\begin{enumerate}
\item[\textbf{1.}] \textbf{Scaling:} By Proposition \ref{Prop Spectra} \ref{Scalling}, it suffices to consider the case where $\bb$ lies on the unit circle, \emph{i.e.}, $\bb=e^{i\theta}$ with $\theta\in (-\pi,\pi]$.
\item[\textbf{2.}] \textbf{Symmetry:} Using Proposition \ref{Prop Spectra} \ref{Sym 1}, this further reduces the analysis to  $\bb=e^{i\theta}$ with $\theta \in \left(-\frac{\pi}{2}, \frac{\pi}{2}\right]$.
\item[\textbf{2.}] \textbf{Reflection Symmetry:} Finally, by Proposition \ref{Prop Spectra} \ref{Sym 2}, it suffices to restrict to the case $\bb=e^{i\theta}$ with $\theta\in \left[0,\frac{\pi}{2}\right]$.
\end{enumerate}
In summary, it is enough to analyze~$\mathscr{L}_{\bb}$ on the first-quadrant arc of the unit circle.

	\section{When the magnetic field is non-real and non-imaginary}\label{Sec Complex b}
This section is devoted to the proof of Theorem \ref{Theo Spectrum} \ref{b complex}. From Section \ref{Subsec Spectra Reduc},  we may assume that $\bb=e^{i\theta}$, with $\theta\in\left(0,\frac\pi2\right)$.
Thanks to the partial Fourier transform \eqref{Partial Fourier 2}, we get
\begin{align*}
	&\widehat{\mathscr{L}}_\bb = \mathcal{F}_{x_{2}\mapsto \xi_{2}} \mathscr{L}_{\bb} \mathcal{F}_{x_{2}\mapsto \xi_{2}}^{-1}= -\partial_{x}^2 + (\bb x-\xi_2)^2,\\
	&\textup{Dom}(\widehat{\mathscr{L}}_\bb)=\left\{ g\in L^2(\R^2) : (-\partial_{x}^2+(\bb x-\xi_2)^2)g\in L^2(\R^2)  \right\}\,.
\end{align*}

\subsection{Weyl sequence construction}
In this section, we show that the spectrum of $\widehat{\mathscr{L}}_\bb$ is the whole complex plane.  Let $\theta\in\left(0,\frac\pi2\right)$ and $\lambda\in\C$ be fixed.

%We will prove this theorem by showing that for each $\lambda\in \C$, there exists a singular Weyl sequence $f_{n}$ corresponding to $\lambda$, \emph{i.e.,} $\Vert f_{n} \Vert =1$ and 
%	\[\lim_{n\to +\infty} \left\Vert \left(\widehat{\mathscr{L}}_\bb-\lambda\right) f_{n}\right\Vert =0.\]
%Let $\varepsilon_{\theta}$ be a fixed number which is assumed to be
%\begin{equation}\label{epsilon theta}
%0<\varepsilon_{\theta} < \frac{2\sin(\theta)+\cos(\theta)-\sqrt{\cos^2(\theta)+4\sin^2(\theta)}}{\sqrt{\cos^2(\theta)+4\sin^2(\theta)}-\cos(\theta)},
%\end{equation}
%and we define
%\begin{align*}
%&a_{\theta}= \frac{(1+\varepsilon_{\theta})(1+\cos(\theta))-\sin(\theta)}{(1+\varepsilon_{\theta})(1+\cos(\theta))+\sin(\theta)}, &&b_{\theta}= \frac{(1+\varepsilon_{\theta})(1+\cos(\theta))+\sin(\theta)}{(1+\varepsilon_{\theta})(1+\cos(\theta))-\sin(\theta)}.\\
%&c_{\theta}= \frac{a_{\theta}+b_{\theta}}{2}, &&d_{\theta}=\frac{b_{\theta}-a_{\theta}}{4}.
%\end{align*}
%It can be seen that $0<a_{\theta}<c_{\theta}<b_{\theta}$ and $d_{\theta}>0$.

Consider
\[t_{\theta}= \frac{1}{\cos(\theta)} \,, \]
choose $d \in (0,\tan(\theta))$ and take a function $\varphi_{\theta}\in C_{c}^{\infty}(\R)$ such that $0\leq \varphi_{\theta} \leq 1$ and 
\begin{equation}
	\varphi_{\theta}(t)=\left\{ \begin{aligned}
		&1\qquad \text{if }t \in \left[t_{\theta}-d,t_{\theta}+d\right]\eqqcolon J_{\theta}'\,,\\
		&0 \qquad \text{if }t\notin \left[t_{\theta}-2d,t_{\theta}+2d\right]\eqqcolon J_{\theta}\,.
	\end{aligned} \right.
\end{equation}
Then, we define the following family of functions:
\begin{equation}\label{Psi n}
	\Psi_{n}(x,\xi_{2}) = \mathds{1}_{[n-1,n+1]}(\xi_{2}) \varphi_{\theta}\left(\frac{x}{\xi_{2}}\right)\, u(x,\xi_{2})\,,
\end{equation}
where
\begin{equation}\label{func u}
	u(x,\xi_{2})= e^{-\frac{1}{4}Z^2(x,\xi_{2})}Z(x,\xi_{2})^{\frac{\lambda}{2\bb}-\frac{1}{2}}\,,\qquad \text{with } Z(x,\xi_{2})=\sqrt{2\bb} \left( x-\frac{\xi_{2}}{\bb}\right)\,.
\end{equation}
	For each $n\in \N$, observe that
\[ \operatorname{Supp}(\Psi_{n})\subset \left\{ (x,\xi_{2})\in \R^2: x \in [(t_{\theta}-2d)\xi_{2},(t_{\theta}+2d)\xi_{2}],\hspace{0.2 cm} \xi_{2} \in [n-1,n+1]\right\}\,,\]
which is a bounded set in $\R^2$, and on which 
$(x,\xi_2) \mapsto \varphi\left(\frac{x}{\xi_{2}}\right)u(x,\xi_{2})$ is smooth, thus $\Psi_{n}\in L^2(\R^2)$. 

In the rest of this section, we prove that $(\Psi_n)$ is a Weyl sequence for the operator $\widehat{\mathscr{L}}_\bb$ associated with $\lambda$. In other words, we are going to prove the validity of the limit~\eqref{eq.Weyl} below. 
To do so, we start by bounding $\|\Psi_n\|$ from below when $n$ is large.
\begin{lemma}\label{Lem Lower Bound}
We have, as $n\to+\infty$,
	\begin{equation*}
		\Vert \Psi_{n} \Vert^2 \gtrsim n^{\Re \left(\frac{\lambda}{\bb}\right)-1} e^{\frac{\sin^2\theta}{\cos\theta}(n-1)^2}\,.
	\end{equation*}

%	Constant in the notation $\gtrsim$ only depends on $\frac{\lambda}{\bb}$ and is independent of $n$.
\end{lemma}
\begin{proof}
Thanks to the expressions
	\begin{equation*}
		\Re Z^2(x,\xi_{2}) = 2\left[\cos(\theta) x^2 - 2x\xi_{2} + \cos(\theta) \xi_{2}^2\right],\qquad \vert Z(x,\xi_{2}) \vert^2 = 2\left[  x^2 - 2\cos(\theta)x\xi_{2} +  \xi_{2}^2\right]\,,
	\end{equation*}
	and \eqref{Power Modul}, we have
	\begin{align*}
		&\int_{\R^2} \vert \Psi_{n}(x,\xi_{2}) \vert^2 \, \dd x \dd \xi_{2}\\
		= &\int_{n-1}^{n+1}  \int_{(t_{\theta}-2d)\xi_{2}}^{(t_{\theta}+2d)\xi_{2}} \left\vert \varphi\left(\frac{x}{\xi_{2}}\right)\right\vert^2\left\vert u(x,\xi_{2}) \right\vert^2\, \dd x \dd \xi_{2}\\
		\gtrsim & \int_{n-1}^{n+1} \int_{\left(t_{\theta}-d\right)\xi_{2}}^{\left(t_{\theta}+d\right)\xi_{2}}e^{-(\cos(\theta) x^2 - 2x\xi_{2} + \cos(\theta) \xi_{2}^2)}\left( x^2 - 2\cos(\theta)x\xi_{2} +  \xi_{2}^2 \right)^{\Re \left(\frac{\lambda}{2\bb}\right)-\frac{1}{2}}\, \dd x \dd \xi_{2}\,.
	\end{align*}
 By the change of variable $x=t\xi_{2}$, it leads to
	\[\int_{\R^2} \vert \Psi_{n}(x,\xi_{2}) \vert^2 \, \dd x \dd \xi_{2} \gtrsim \int_{n-1}^{n+1} \xi_{2}^{\Re \left(\frac{\lambda}{\bb}\right)} \int_{t_{\theta}-d}^{t_{\theta}+d} e^{\xi_{2}^2\, p(t)} q(t) \, \dd t \dd \xi_{2}\,,\]
	where
		\begin{equation}\label{p}
		p(t)= - \cos(\theta) t^2+2t-\cos(\theta)\,,\quad 		q(t)= (t^2-2\cos(\theta)t+1)^{\Re \left(\frac{\lambda}{2\bb}\right)-\frac{1}{2}}\,.
	\end{equation} 
	Since $p$ attains its maximum at $t_{\theta}$ (with $p(t_\theta)=\frac{\sin^2(\theta)}{\cos(\theta)}>0$), the Laplace method yields, as $\xi_2\to+\infty$,
	\[ \int_{t_{\theta}-\frac{1}{2}}^{t_{\theta}+\frac{1}{2}} e^{\xi_{2}^2\, p(t)} q(t)\, \dd t \sim \sqrt{\frac{\pi}{\cos(\theta)}}\, q(t_\theta) \frac{e^{p(t_\theta)\xi_{2}^2}}{\xi_{2}}\,.\]
	Since $n-1<\xi_{2}<n+1$, we have
	\[\xi_{2}^{\Re \left(\frac{\lambda}{\bb}\right)-1} \gtrsim n^{\Re \left(\frac{\lambda}{\bb}\right)-1} \qquad \text{ and }\qquad p(t_\theta)\xi_{2}^2> p(t_\theta) (n-1)^2\,.\]
	The conclusion follows.
\end{proof}

\begin{proposition}\label{Theo Spectrum L b Hat}
We have
\begin{equation}\label{eq.Weyl}
\lim_{n\to+\infty}\frac{\left\Vert \left(\widehat{\mathscr{L}}_\bb-\lambda\right) \Psi_{n}\right\Vert^2 }{\left\Vert \Psi_{n}\right\Vert^2 }=0\,.
\end{equation}
\end{proposition}

\begin{proof}
By using the definition of the function $u$ in \eqref{func u}, 
a straightforward computation gives
	\[ \left(-\partial_{x}^2 + (\bb x-\xi_2)^2 - \lambda\right)u(x,\xi_{2}) =\frac{C_{\lambda,\bb}}{Z^2(x,\xi_{2})}u(x,\xi_{2})\,,\qquad C_{\lambda,\bb}= -\frac{\lambda^2}{2}+2\lambda-\frac{3\bb}{2}\,.\]
	Then, it yields that
	\begin{align*}
		& \left(\widehat{\mathscr{L}}_\bb-\lambda\right) \Psi_{n}(x,\xi_{2})\\
		=& \mathds{1}_{[n-1,n+1]}(\xi_{2})\left[- \frac{1}{\xi_{2}^2} \varphi''\left(\frac{x}{\xi_{2}}\right)\, u(x,\xi_{2})- 2\frac{1}{\xi_{2}} \varphi'\left(\frac{x}{\xi_{2}}\right) \,\partial_{x}u(x,\xi_{2})\right]+\frac{C_{\lambda,\bb}}{Z^2(x,\xi_{2})} \Psi_{n}(x,\xi_{2})\,.
	\end{align*}
	From the inequalities 
	\[ 2 \sin^2(\theta) \xi_{2}^2 \leq |Z(x,\xi_{2})|^2\leq 2(|x|+|\xi_{2}|)^2,\]
	we deduce that, for all $ n\gtrsim 1$, all $\xi_{2} \in [n-1,n+1]$, and all $x\in \xi_{2} J_{\theta}$,  
	\begin{equation}\label{Estimate Z}
		|Z(x,\xi_{2})| \approx |\xi_{2}| \approx n\,.
	\end{equation}
Hence, it leads to
	\begin{equation*}
		\frac{\left\Vert \left(\widehat{\mathscr{L}}_\bb-\lambda\right) \Psi_{n}\right\Vert^2 }{\left\Vert \Psi_{n}\right\Vert^2 }\lesssim \left[ \frac{\Vert T_{1,n} \Vert^2 + \Vert T_{2,n} \Vert^2}{\Vert \Psi_{n} \Vert^2} + \frac{1}{n^4}\right]\,,
	\end{equation*}
	where
	\begin{align*}
	T_{1,n}(x,\xi_{2}) &=  \mathds{1}_{[n-1,n+1]}(\xi_{2})\, \frac{1}{\xi_{2}^2} \varphi''\left(\frac{x}{\xi_{2}}\right) u(x,\xi_{2})\,,\\
	T_{2,n}(x,\xi_{2}) &=  \mathds{1}_{[n-1,n+1]}(\xi_{2})\, \frac{1}{\xi_{2}} \varphi'\left(\frac{x}{\xi_{2}}\right) \partial_{x}u(x,\xi_{2})\,.
	\end{align*}
	Notice that, for $k\in \{1,2\}$,
	\[ \operatorname{Supp} T_{k,n} \subset \left\{ (x,\xi_{2})\in \R^2: x \in \xi_{2}\left(J_{\theta}\setminus J_{\theta}'\right),\, \xi_{2} \in [n-1,n+1]\right\}\,.\]
	Let us start with $T_{1,n}$. By using the change of variable $x=\xi_{2} t$, we get
	\begin{align*}
		&\int_{\R^2} \vert T_{1,n} \vert^2 \, \dd x \dd \xi_{2}\\
		= & \int_{n-1}^{n+1} \frac{1}{\xi_{2}^4} \int_{\xi_{2}\left(J_{\theta}\setminus J_{\theta}'\right)} \left\vert \varphi''\left(\frac{x}{\xi_{2}}\right)\right\vert^2 \left\vert u(x,\xi_{2}) \right\vert^2\, \dd x\dd \xi_{2}\\
		\approx & \int_{n-1}^{n+1} \frac{1}{\xi_{2}^4}  \int_{\xi_{2}\left(J_{\theta}\setminus J_{\theta}'\right)}\left\vert \varphi''\left(\frac{x}{\xi_{2}}\right)\right\vert^2 e^{-\cos(\theta) x^2 +2x\xi_{2} - \cos(\theta) \xi_{2}^2}\\
		& \hspace{3 cm}\times\left( x^2 - 2\cos(\theta)x\xi_{2} +  \xi_{2}^2 \right)^{\Re \left(\frac{\lambda}{2\bb}\right)-\frac{1}{2}}\, \dd x \dd \xi_{2}\\
		\approx &\int_{n-1}^{n+1} \xi_{2}^{\Re \left(\frac{\lambda}{\bb}\right)-4 } \int_{J_{\theta}\setminus J'_{\theta}} \vert\varphi''(t)\vert^2 e^{\xi_{2}^2\, p(t)} q(t) \, \dd t \dd \xi_{2}\,,
	\end{align*}
	where $p(t)$ and $q(t)$ are given in \eqref{p}.	Since $p$ is a concave polynomial of degree two and attains its maximum at $t_{\theta} \in (t_{\theta}-d,t_{\theta}+d)$, we have
	\[p(t)\leq \kappa_{\theta}=   p\left(t_{\theta}+d\right)<p(t_\theta)\,, \qquad \forall t\in J_{\theta}\setminus J'_{\theta}\,.\]
	Note that $\kappa_{\theta}=\cos\theta((\tan^2\theta-d^2)$.
	Thus, we get
	\begin{align*}
		\int_{\R^2} \vert T_{1,n} \vert^2 \, \dd x \dd \xi_{2}\lesssim &\int_{n-1}^{n+1} \xi_{2}^{\Re \left(\frac{\lambda}{\bb}\right)-4 }  e^{\kappa_{\theta}\xi_{2}^2} \, \dd \xi_{2}	\lesssim n^{\Re \left(\frac{\lambda}{\bb}\right)-4} e^{\kappa_{\theta}(n+1)^2}\,, \qquad \text{as }		n\to+\infty\,.
	\end{align*}
	For the term $T_{2,n}$, we observe that
	\[ \partial_{x}u(x,\xi_{2}) =\sqrt{2\bb}\left[ -\frac{Z(x,\xi_{2})}{2}+\left(\frac{\lambda}{2\bb}-\frac{1}{2}\right)\frac{1}{Z(x,\xi_{2})}\right]u(x,\xi_{2})\,.\]
	Thanks to \eqref{Estimate Z} and arguing as above, we get	
	\[	\int_{\R^2} \vert T_{2,n} \vert^2 \, \dd x \dd \xi_{2}\lesssim 
	n^{\Re \left(\frac{\lambda}{\bb}\right)} e^{\kappa_{\theta}(n+1)^2},\qquad \text{as }n\to+\infty\,. \]
	From the estimates of $T_{1,n}$ and $T_{2,n}$ above and Lemma \ref{Lem Lower Bound}, we deduce that
	\begin{align*}
		\frac{\left\Vert \left(\widehat{\mathscr{L}}_\bb-\lambda\right) \Psi_{n}\right\Vert^2 }{\left\Vert \Psi_{n}\right\Vert^2 }\lesssim \left[n e^{\kappa_{\theta}(n+1)^2-p(t_{\theta})(n-1)^2}+\frac{1}{n^4}\right], \qquad \text{as }n\to +\infty\,.
	\end{align*}
	Since $p(t_{\theta})-\kappa_{\theta}>0$, the right-hand side goes to zero as $n\to+\infty$. 
\end{proof}
Proposition \ref{Theo Spectrum L b Hat} establishes that $\lambda\in\mathrm{Spec}(\widehat{\mathscr{L}}_\bb)$. Since the support of $\Psi_{n}$ escapes to infinity in the $\xi_{2}$ direction, it yields that $f_{n}= \frac{\Psi_{n}}{\Vert \Psi_{n} \Vert}$ weakly converges to zero. Applying \cite[Thm.~9.1.3(i)]{Edmunds-Evans18}, we obtain the following result:
\begin{corollary}\label{Cor Weakly Conv}
	We have $\operatorname{Spec}_{\textup{ess,2}}(\widehat{\mathscr{L}}_\bb) = \C$.
	\end{corollary}

\subsection{Complex Landau levels}
Let us show that the elements of  
$\Lambda_\bb = \{ (2k+1)\bb: k\in \N_{0}\}$ 
are the eigenvalues of $\widehat{\mathscr{L}}_\bb$
(recall that $\Re\bb > 0$ in this section). 
For that purpose, we recall that the Hermite functions $\psi_{k}\in L^2(\R)$ are defined by
\[ \psi_{k}(x)= c_{k} e^{-\frac{x^2}{2}} H_{k}(x)\,,\]
where $H_{k}$ is the $k$-th Hermite polynomial and $c_k>0$ is a normalizing constant (see  \cite[Thm.~6.2]{Zworski12}). They satisfy
\begin{equation}\label{Harmonic Eq}
	\left(-\frac{\dd^2}{\dd x^2} +x^2  \right) \psi_{k} =(2k+1) \psi_{k}\,.
\end{equation}
This suggests to consider the family
	\[h_{k\ell}(x,\xi_2)= \psi_{k}\left(\sqrt{\bb}\left(x-\frac{\xi_2}{\bb}\right)\right)\psi_\ell\left(\frac{\xi_{2}}{\sqrt{\cos\theta}}\right)\,,\]
which satisfies
	\begin{equation}\label{Eigenvalue}
	\widehat{\mathscr{L}}_\bb h_{k\ell} = \bb \,(2k+1)  h_{k\ell}\,,\quad 	\widehat{\mathscr{L}}_\bb^* \overline{h}_{k\ell} = \overline{\bb} (2k+1)  \overline{h}_{k\ell}\,,
\end{equation}
where we used Proposition \ref{Prop Adjoint} for the second equality. The presence of $\sqrt{\cos\theta}$ is here to ensure that $h_{k\ell}\in L^2(\R^2)$. Indeed, we have
\begin{equation*}
	\left\vert h_{k\ell}(x,\xi_{2}) \right\vert^2 = |c_{k}c_\ell|^2  \left\vert H_{k}\left(\sqrt{\bb}\left(x-\frac{\xi_2}{\bb}\right)\right)H_{\ell}\left(\frac{\xi_2}{\sqrt{\cos\theta}}\right)\right\vert^2 e^{- \left(\cos(\theta) x^2 -2 x\xi_{2}+\cos(\theta)\xi_{2}^2\right) } e^{-\frac{\xi_2^2}{\cos\theta}}\,,
\end{equation*}
which is integrable since
\[\cos(\theta) x^2 -2 x\xi_{2}+\cos(\theta)\xi_{2}^2+\frac{\xi^2_2}{\cos\theta}=\cos\theta\left[\left(x-\frac{\xi_2}{\cos\theta}\right)^2+\xi_2^2\right]\,.\]
This shows that
\[
\Lambda_\bb
\subset \operatorname{Spec}_{\textup{p}}(\widehat{\mathscr{L}}_\bb)\,.
\]

\begin{lemma}\label{lem.total}
	The family $(h_{k\ell})_{(k,\ell)\in\N_0^2}$ is a total family in $L^2(\R^2)$.	
\end{lemma}
\begin{proof}
The proof is standard. We recall the main steps for completeness. Take $\psi\in L^2(\R^2)$ such that $\psi$ is orthogonal to all the $h_{k\ell}$. Then, $\psi=\psi(x_1,x_2)$ is also orthogonal to all the family 
\[\left(x_1^k x_2^\ell e^{-\frac12Q(x_1,x_2)}\right)_{(k,\ell)\in\N_0^2}\,,\quad\mbox{ where }\quad Q(x_1,x_2)=\bb\left(x_1-\frac{x_2}{\bb}\right)^2+\frac{x_2^2}{\cos\theta}\,.\] 
We let, for all $\xi\in\R^2$,
\[F(\xi)=\int_{\R^2}\psi(x) e^{-ix\cdot\xi} e^{-\frac12Q(x)}\mathrm{d} x\,.\]
By using that $\Re Q$ is a positive definite quadratic form, the function $F$ extends to a holomorphic function on a strip about $\R^2$. By using the dominate convergence theorem and the aforementioned orthogonality, we get $\partial^\alpha_{\xi}F(0)=0$, for all $\alpha\in\N_0^2$. This shows that $F$ is zero and so is $\psi$ thanks to the inverse Fourier transform.
	\end{proof}

In fact, the Landau levels are the only elements of the point spectrum.
\begin{proposition}\label{Theo C Spectrum L b Hat}
	Let $\bb=e^{i\theta}$ with $\theta\in\left(0,\frac\pi2\right)$, then
	\[ \operatorname{Spec}_{\textup{p}}(\widehat{\mathscr{L}}_\bb)\subset 
	\Lambda_\bb
	\,.
	\]
\end{proposition}
\begin{proof}
	Consider $\lambda$ in the point spectrum and $\psi$ a corresponding normalized eigenfunction. We have $(\widehat{\mathscr{L}}_\bb-\lambda)\psi=0$ so that, for all $(k,\ell)\in\N^2_0$, $\langle(\widehat{\mathscr{L}}_\bb-\lambda)\psi,\overline{h}_{k\ell}\rangle=0$.
	By using the adjoint and \eqref{Eigenvalue}, we get $\langle\psi,((2k+1)\overline{\bb}-\overline{\lambda})\overline{h}_{k\ell}\rangle=0$. Thus, for all $(k,\ell)\in\N^2_0$, $((2k+1)\bb-\lambda)\langle\psi,\overline{h}_{k\ell}\rangle=0$. If for all $k\in\N_0$, we have $(2k+1)\bb-\lambda\neq 0$, then Lemma \ref{lem.total} shows that $\psi=0$ which is a contradiction.
\end{proof}

			\section{When the magnetic field is purely imaginary}\label{Sec Imag b}
			In this section, we prove 
			Theorem \ref{Theo Spectrum}~\ref{b imaginary}. Thanks to Proposition \ref{Prop Spectra}, it is enough to study the operator when $\bb=i$. 
			That is why we consider~$\mathscr{L}_{i}$ only.

It will be convenient to use a change of gauge to cancel the term $x_1^2$:
\[\widetilde{\mathscr{L}}_i=e^{-ix_1^2/2}\mathscr{L}_i e^{ix_1^2/2}= (-i\partial_{x_{1}}+x_1)^2 + (-i\partial_{x_{2}}- ix_{1})^2=-\partial_{x_1}^2-\partial_{x_2}^2+(x_1 D_{x_1}+D_{x_1}x_1)-2ix_1 D_{x_2}\,.\]
By using the Fourier transform $\mathcal{F}_{x\mapsto \xi}$ given in \eqref{Fourier}, we get the \emph{first order} differential operator
	\begin{equation}\label{L i Hat}
	\begin{aligned}
		&\widehat{\mathscr{L}}_{i}= \mathcal{F}_{x\mapsto \xi}\,\widetilde{\mathscr{L}}_{i}\,\mathcal{F}_{x\mapsto \xi}^{-1}=2(i\xi_{1}+\xi_{2})\partial_{\xi_{1}}+ \xi_{1}^2+\xi_{2}^2+i\,,\\
		&\textup{Dom}(\widehat{\mathscr{L}}_{i})=\left\{ g\in L^2(\R^2) : \left[ 2(i\xi_{1}+\xi_{2})\partial_{\xi_{1}}+ \xi_{1}^2+\xi_{2}^2\right]g\in L^2(\R^2)  \right\}\,.
	\end{aligned}
	        \end{equation}
The purely imaginary case is special since there is no point spectrum any more. 
Indeed, take $\lambda\in \C$, suppose that $\psi$ is an eigenfunction corresponding to $\lambda$ of $\widehat{\mathscr{L}}_{i}$, \emph{i.e.}, $\psi\in \operatorname{Dom}(\widehat{\mathscr{L}}_{i})$ and
\[ 2(i\xi_{1}+\xi_{2})\partial_{\xi_{1}}\psi(\xi_{1},\xi_{2})+ \left(\xi_{1}^2+\xi_{2}^2+i-\lambda\right)\psi(\xi_{1},\xi_{2})=0\,.\]
By fixing $\xi_{2}\neq 0$ and solving the ordinary differential equation with respect to $\xi_{1}$, we obtain a solution in the form
\begin{equation}\label{function u}
	u(\xi_{1},\xi_{2})=e^{-\frac{1}{2}\xi_{1}\xi_{2}+i\frac{\xi_{1}^2}{4}}(\xi_{1}-i\xi_{2})^{-\frac{1+i\lambda}{2}}\,,
\end{equation}
and thus, for some constant $C(\xi_2)$, we have, for all $\xi_1\in\R$, $\psi(\xi)=C(\xi_2)u(\xi)$. For all $\xi_2\neq 0$, $u(\cdot, \xi_2)$ does not belong to $L^2(\R)$. Thus, $\psi=0$.

However, the spectrum is still the whole complex plane. To see this, it is sufficient to consider $\lambda$ with $\Im\lambda\geq 0$ and to construct an appropriate Weyl sequence. 
Indeed, in the purely imaginary case, one has the extra symmetry
$T\mathscr{L}_{i}=  \mathscr{L}_{i} T$, where $T\psi = \overline{\psi}$. 
Consequently,
$\lambda \in \operatorname{Spec}_{\textup{ess,2}}(\mathscr{L}_{i}) \Longleftrightarrow  \overline{\lambda} \in \operatorname{Spec}_{\textup{ess,2}}(\mathscr{L}_{i})$.

We fix $\lambda\in\C$ with $\Im\lambda\geq 0$.	Choosing $\alpha\in (1,2)$ and we set
\[ \Psi_{n}(\xi_{1},\xi_{2}) = \mathds{1}_{[1,2]}(n\xi_{2}) \varphi\left(\frac{\xi_{1}}{n^{\alpha}\xi_{2}}\right)u(\xi_{1},\xi_{2})\,,\]
where $\varphi \in  C^{\infty}(\R)$ is such that 
\begin{align*}
	\varphi(s)=\left\{ \begin{aligned}
		& 1  \qquad \text{if }s\geq 2\,,\\
		& 0  \qquad \text{if }s\leq 1\,.\\
	\end{aligned}\right.
\end{align*}

\begin{lemma}
For all $n\in \N$, we have $\Psi_n\in L^2(\R^2)$.
\end{lemma}

\begin{proof}
	Fix $n\in \N$, by using the change of variable $\xi_{1}=s \xi_{2}$, we have
	\begin{align*}
		\int_{\R^2} \vert \Psi_{n}(\xi_{1},\xi_{2}) \vert^2\, \dd \xi_{1}\dd \xi_{2}
		=& \int_{\frac{1}{n}}^{\frac{2}{n}}  \int_{n^{\alpha}\xi_{2}}^{+\infty} \left\vert \varphi\left(\frac{\xi_{1}}{n^{\alpha}\xi_{2}}\right) \right\vert^2 \vert u(\xi_{2},\xi_{2}) \vert^2\, \dd \xi_{1} \dd \xi_{2}\\
		\lesssim & \int_{\frac{1}{n}}^{\frac{2}{n}}  \int_{n^{\alpha}\xi_{2}}^{+\infty} e^{-\xi_{1}\xi_{2}}(\xi_{1}^2+\xi_{2}^2)^{\frac{1}{2}(\Im \lambda-1)}\, \dd \xi_{1} \dd \xi_{2}\\
		= & \int_{\frac{1}{n}}^{\frac{2}{n}} \xi_{2}^{\Im \lambda} \int_{n^{\alpha}}^{+\infty} e^{-\xi_{2}^2 s}(s^2+1)^{\frac{1}{2}(\Im \lambda-1)}\, \dd s \dd \xi_{2}\\
		\leq & \left(\int_{\frac{1}{n}}^{\frac{2}{n}} \xi_{2}^{\Im \lambda} \, \dd \xi_{2}\right) \left( \int_{0}^{+\infty} e^{-s/n^2}(s^2+1)^{\frac{1}{2}(\Im \lambda-1)}\,\dd s \right)<+\infty\,.
	\end{align*}
	\end{proof}

\begin{lemma}\label{lem.L2normPsini}
We have
\begin{equation}
 \Vert \Psi_{n} \Vert^2 \gtrsim \left\{\begin{aligned}
 & n^{\Im \lambda-1} &&\text{ if } \Im\lambda>0,\\
 & \ln(n) n^{-1}      &&\text{ if } \Im\lambda=0,
\end{aligned}  \right.
\end{equation}
as $n\to+\infty$.
\end{lemma}
\begin{proof}
We have
	\begin{align*}
		\Vert \Psi_{n} \Vert^2 	\geq & \int_{\frac{1}{n}}^{\frac{2}{n}}  \int_{2n^{\alpha}\xi_{2}}^{+\infty} e^{-\xi_{1}\xi_{2}}(\xi_{1}^2+\xi_{2}^2)^{\frac{1}{2}(\Im \lambda-1)}\, \dd \xi_{1} \,\dd \xi_{2}\\
		= & \int_{\frac{1}{n}}^{\frac{2}{n}} \xi_{2}^{\Im \lambda} \int_{2n^{\alpha}}^{+\infty} e^{-\xi_{2}^2 s}(s^2+1)^{\frac{1}{2}(\Im \lambda-1)}\, \dd s\, \dd \xi_{2} \qquad &&\left( \text{change }\xi_{1}=\xi_{2}s\right)\\
		\approx & \int_{\frac{1}{n}}^{\frac{2}{n}} \xi_{2}^{\Im \lambda} \int_{2n^{\alpha}}^{+\infty} e^{-\xi_{2}^2 s}s^{\Im \lambda-1}\, \dd s \,\dd \xi_{2} \qquad &&\left( \text{use }s^2\leq s^2+1\leq 2s^2, \forall s\geq 1\right)\\
		=& \int_{\frac{1}{n}}^{\frac{2}{n}} \xi_{2}^{-\Im \lambda} \int_{2n^{\alpha}\xi_{2}^2}^{+\infty} e^{-t} t^{\Im \lambda-1}\,\dd t\, \dd \xi_{2} &&\left( \text{change } t=\xi_{2}^2s\right)\\
		=&\int_{\frac{1}{n}}^{\frac{2}{n}} \xi_{2}^{-\Im \lambda} \Gamma(\Im \lambda, 2n^{\alpha}\xi_{2}^2)\, \dd \xi_{2}\,,
	\end{align*}
	where 
	$
	\Gamma(a,x)
	= \int_{x}^{+\infty} t^{a-1}e^{-t}\, \dd t
	$ is the incomplete Gamma function. 
	Since $\alpha<2$ and $2n^{\alpha}\xi_{2}^2\leq 8n^{\alpha-2}$ for all $\xi_{2}\in \left[\frac{1}{n},\frac{2}{n}\right]$, it implies that $\lim_{n\to+\infty} 2n^{\alpha}\xi_{2}^2=0$. Now we consider two cases. When $\Im\lambda>0$,  we have, by dominated convergence theorem,
	$$\lim_{n\to +\infty}\Gamma(\Im \lambda, 2 n^{\alpha}\xi_{2}^2)=\Gamma(\Im\lambda)\,.$$
Thus, it leads to
			\[ \Vert \Psi_{n} \Vert^2 \gtrsim \int_{\frac{1}{n}}^{\frac{2}{n}} \xi_{2}^{-\Im \lambda}\, \dd \xi_{2} \approx n^{\Im \lambda-1}\,.\]
	When $\Im\lambda=0$, we have
			\[ \Gamma(0,2n^{\alpha}\xi_{2}^2) \approx -\ln(2n^{\alpha}\xi_{2}^2) \approx -\ln(n^{\alpha-2})\approx \ln(n),\qquad \text{for all } \xi_{2}\in \left[\frac{1}{n},\frac{2}{n}\right] ,\text{ as }n\to +\infty\,,\]
and then
		\[ \Vert \Psi_{n} \Vert^2 \gtrsim \int_{\frac{1}{n}}^{\frac{2}{n}}\ln(n)\, \dd \xi_{2} \approx \ln(n) n^{-1}\,.\]
	\end{proof}

\begin{proposition}\label{prop.Weyli}
We have
\[\Vert (\widehat{\mathscr{L}}_{i}-\lambda)\Psi_{n} \Vert^2\lesssim n^{-\Im \lambda-1-2\alpha}\,,\]
as $n\to+\infty$.
\end{proposition}
\begin{proof}
Notice that
\begin{align*}
	(\widehat{\mathscr{L}}_{i}-\lambda)\Psi_{n}(\xi_{1},\xi_{2}) = 2\chi_{[1,2]}(n\xi_{2}) \frac{1}{n^{\alpha}} \left(i\frac{\xi_{1}}{\xi_{2}}+1\right)\varphi'\left(\frac{\xi_{1}}{n^{\alpha}\xi_{2}}\right)u(\xi_{1},\xi_{2}).
\end{align*}
Then, we have
\begin{align*}
	\Vert (\widehat{\mathscr{L}}_{i}-\lambda)\Psi_{n} \Vert^2 = & 4\int_{\frac{1}{n}}^{\frac{2}{n}}  \int_{n^{\alpha}\xi_{2}}^{2n^{\alpha}\xi_{2}} \frac{1}{n^{2\alpha}}\left(\left(\frac{\xi_{1}}{\xi_{2}}\right)^2+1\right)\left\vert \varphi'\left(\frac{\xi_{1}}{n^{\alpha}\xi_{2}}\right) \right\vert^2 \vert u(\xi_{2},\xi_{2}) \vert^2\, \dd \xi_{1} \dd \xi_{2}\\
	\lesssim & \frac{1}{n^{2\alpha}}\int_{\frac{1}{n}}^{\frac{2}{n}}  \int_{n^{\alpha}\xi_{2}}^{2n^{\alpha}\xi_{2}}e^{-\xi_{1}\xi_{2}}(\xi_{1}^2+\xi_{2}^2)^{\frac{1}{2}(\Im \lambda-1)}\, \dd \xi_{1} \dd \xi_{2}\\
	= &\frac{1}{n^{2\alpha}} \int_{\frac{1}{n}}^{\frac{2}{n}} \xi_{2}^{\Im \lambda} \int_{n^{\alpha}}^{2n^{\alpha}} e^{-\xi_{2}^2 s}(s^2+1)^{\frac{1}{2}(\Im \lambda-1)}\, \dd s \dd \xi_{2}\\
	\lesssim &\frac{1}{n^{2\alpha}} \int_{\frac{1}{n}}^{\frac{2}{n}} \xi_{2}^{\Im \lambda} \,\dd \xi_{2}\approx \, n^{-\Im \lambda-1-2\alpha}.
\end{align*}	
\end{proof}
By setting $f_{n}= \frac{\Psi_{n}}{\Vert \Psi_{n} \Vert}$, using Lemma \ref{lem.L2normPsini} and Proposition \ref{prop.Weyli}, we get
	\[	\lim_{n\to+\infty}\Vert (\widehat{\mathscr{L}}_{i}-\lambda)f_{n} \Vert=0\,.\]
	This shows that
		\[ \{\lambda\in \C: \Im \lambda\geq 0\}\subset \operatorname{Spec}(\widehat{\mathscr{L}}_{i})=\operatorname{Spec}(\mathscr{L}_{i})\,.\]
Since $\alpha>1$ and $\operatorname{Supp}(f_{n})\subset [n^{\alpha-1},+\infty] \times \left[ \frac{1}{n}, \frac{2}{n}\right]$, support of $f_{n}$ moves to infinity in $\xi_{1}$ direction. Hence, $f_{n}$ weakly converges to $0$ in $L^2(\R^2)$. 

\section*{Acknowledgements}
D.K. and T.N.D. were supported by the EXPRO grant number 20-17749X of the Czech Science Foundation (GA\v{C}R).
			\bibliographystyle{amsplain}
			\bibliography{Ref11}
			\end{document}